\documentclass{jms_my}
\usepackage{color}
\paperTitle{Integral representation of the Mittag-Leffler function}
\articleColonName{Integral representation of the ML function}
\authorsShort{V.\,V.~Saenko}
\authorsFull{V.\,V. Saenko\first}
\addAuthorInfo{Ulyanovsk State University, S.P. Kapitsa Research Institute of Technology,   L. Tolstoy St. 42, Ulyanovsk,  Russia, 432017 e-mail: \url{saenkovv@gmail.com}}
\paperAbstract{Generalization of the integral representation of the gamma function has been obtained, which shows that the Hankel contour assumes rotation in the complex plane. The range of admissible values for the contour rotation angle is set.
Using this integral representation, generalization of the integral representation of the Mittag-Leffler function has been obtained that expresses the value of this function in terms of the contour integral.}

\begin{document}
\maketitle

\section{Introduction}
The Mittag-Leffler function is an entire function defined by a power series
\begin{equation*}
E_\rho(z)=\sum_{n=0}^{\infty}\frac{z^n}{\Gamma(1+n/\rho)}, \quad \rho>0,\quad z\in\C,
\end{equation*}
where $\Gamma(x)$ is the gamma function. This function was introduced by Mittag-Leffler in a number of papers \cite{Mittag-Leffler1900,Mittag-Leffler1901,Mittag-Leffler1901a,Mittag-Leffler1902,Mittag-Leffler1905,Mittag-Leffler1920} published between 1902 and 1905 in connection with the development of  his method for summing divergent series. The function $E_{\rho}(z)$ is also called the one-parameter Mittag-Leffler function.  Replacing one in a gamma function argument with an arbitrary parameter $\mu\in\C$ leads to the generalization of the function $E_\rho(z)$ in the case with two parameters
\begin{equation}\label{eq:MLF_gen}
  E_{\rho,\mu}(z)=\sum_{n=0}^{\infty}\frac{z^n}{\Gamma(\mu+n/\rho)},\quad \rho>0,\quad \mu\in\C,\quad z\in\C,
\end{equation}
which is called the two-parameter Mittag-Leffler function. This generalization was obtained in the works by A. Wiman in 1905 \cite{Wiman1905,Wiman1905a} got further development in the works by Humbert and Agarwal \cite{zbMATH03082752,zbMATH03078845,zbMATH03081895} and in the works by M. M. Dzhrbashyan \cite{Dzhrbashian1954b_eng,Dzhrbashyan1954_eng} (see also \cite{Dzhrbashyan1966_eng}, Chapter 3, \S2, 4 ). As we can see, the function $E_{\rho,\mu}(z)$ is connected with $E_\rho(z)$  by the ratio $E_{\rho,1}(z)=E_{\rho}(z)$.

Great interest in the Mittag-Leffler function is primarily shown due to the use of this function to solve differential equations expressed in terms of fractional derivatives. In such problems, the Mittag-Leffler function acts as an eigenfunction of the fractional differentiation operators.  The Mittag-Leffler function is also used to describe the processes of anomalous diffusion \cite{Uchaikin2009, UCHAIKIN2008, Cahoy2010}, in probability theory and mathematical statistics \cite{Korolev2017, Bening2006,Saenko2020c}, as well as in many other fields of science, where differential equations appear in fractional derivatives. In this regard, a lot of attention has been paid to the study of the analytical properties of the Mittag-Leffler function. The main properties of the Mittag-Leffler function were elucidated in the book \cite{Bateman_V3_1955}. A more detailed study of analytic and asymptotic properties of the Mittag-Leffler function was given in the book \cite{Dzhrbashyan1966_eng}. In this book, the integral representation of the Mittag-Leffler function was obtained, expressed in terms of the contour integral. The reader can find more detailed information on the Mittag-Leffler function and its properties in the book \cite{Gorenflo2014} and in review papers  \cite{Diethelm2010,Mathai2008,Gorenflo2019,Popov2013,Rogosin2015}.

The integral representation of the function $E_{\rho,\mu}(z)$ considered in this paper expresses its value in terms of the contour integral. The integral representation of the Mittag-Leffler function is used to calculate the value of this function \cite{Gorenflo2002b,Seybold2009,Haubold2011,Parovik2012_eng}, it gives an opportunity to study the asymptotic behavior of the Mittag-Leffler function \cite{Dzhrbashyan1966_eng}, and distribution of its zeros \cite{Popov2013}. The latter turns out to be important in the theory of the integral Fourier-Laplace transforms with Mittag-Leffler kernel. The integral representation also turns out to be convenient when performing integral transformations in which the Mittag-Leffler function acts as the kernel. For example, in the work \cite{Saenko2020c} the inverse Fourier transform of the characteristic function of fractional-stable distribution was performed, which is expressed through the function $E_\rho(z)$. As a result, the density and distribution function of this probability law were obtained.

There are several forms of writing the integral representation of the function $E_{\rho,\mu}(z)$. Each of these forms differs one from another when taking account of additional properties of the Hankel contour and transforming the integrand. One of the first forms of notation is given in the book  \cite{Bateman_V3_1955} (see \S18.1)
\begin{equation}\label{eq:MLF_int_Bateman}
  E_{\alpha,\beta}(z)=\frac{1}{2\pi i}\int_{\gamma(\varepsilon)}\frac{t^{\alpha-\beta}e^t}{t^\alpha-z}dt,\quad\alpha>0,\quad\beta>0.
\end{equation}
Here the contour $\gamma(\varepsilon)$ represents a self-loop consisting of a half-line $S_1=\{t: \arg t=-\pi, |t|\geqslant\varepsilon\}$, an arc of a circle $C_\varepsilon=\{t: -\pi\leqslant\arg t\leqslant\pi,   |t|=\varepsilon >|z|^{1/\alpha}\} $  and a half-line $ S_2=\{t: \arg t=\pi, |t|\geqslant\varepsilon\}$. The contour $\gamma(\varepsilon)$ is traversed in positive direction. The parameters $\alpha$ and $\beta$ of this representation are connected with the parameters $\rho$ and $\mu$ of the formula (\ref{eq:MLF_gen}) by the relations: $\alpha=1/\rho, \beta=\mu$.

To study the asymptotic properties of the Mittag-Leffler function in the work \cite{Dzhrbashyan1954_eng}  an integral representation of this function was obtained in a more general form \begin{equation}\label{eq:MLF_int_Dzh54}
  E_{\rho,\mu}(z)=\frac{\rho}{2\pi i}\int_{\gamma(\varepsilon,\theta)}\frac{\tau^{\rho(1-\mu)}\exp\{\tau^\rho\}}{\tau-z}d\tau.
\end{equation}
Here the contour of integration $\gamma(\varepsilon,\theta)$ is composed of a half-line $S_1=\{\tau: \arg \tau=-\theta, |\tau|\geqslant\varepsilon\}$, an arc of a circle $C_\varepsilon=\{ \tau: -\theta\leqslant \arg \tau\leqslant\theta, |\tau|=\varepsilon\}$ and a half-line $ S_2=\{\tau: \arg \tau=\theta, |\tau|\geqslant\varepsilon\}$, where $\varepsilon$ is an arbitrary  real number satisfying the condition  $\varepsilon>|z|$, and the parameter  $\theta$ satisfies the conditions: $\frac{\pi}{2\rho}<\theta<\pi$, if $1/2<\rho\leqslant1$ and $\frac{\pi}{2\rho}<\theta<\frac{\pi}{\rho}$, if $\rho>1$. The main difference of the representation (\ref{eq:MLF_int_Dzh54}) from the representation (\ref{eq:MLF_int_Bateman}) consists in the contour of integration. The contour of integration $\gamma(\varepsilon,\theta)$ takes account of the fact that the half-lines $S_1$ and $S_2$ of the Hankel contour in the integral representation for the Euler gamma function can be positioned at an arbitrary angle $\theta$ satisfying the condition $\frac{\pi}{2}<\theta<\pi$ (see for example \cite{Markushevich_V2_1965_eng}, p. 324 or \cite{Markushevich_V2_1968_eng}, p. 315). The integrand (\ref{eq:MLF_int_Dzh54}) is obtained from the integrand (\ref{eq:MLF_int_Bateman}) due to the replacement of the integration variable $\tau=t^\alpha$. Further, representation (\ref{eq:MLF_int_Dzh54}) was included in the monograph \cite{Dzhrbashyan1966_eng} in some changed form.

Further generalization of the integral representation of the Mittag-Leffler function was obtained in the paper \cite{Popov2013}. To this end, the authors showed that the Hankel contour in the integral representation of the Euler gamma function admits further generalization. It has been shown that the half-lines $S_1$ and $S_2$ of the Hankel contours can come out at arbitrary angles $\theta_1$ and $\theta_2$ satisfying the conditions $\theta_1\in (-3\pi/2,-\pi/2)$ and $\theta_2\in (\pi/2,3\pi/2)$ (see lemma~1.1.1 in \cite{Popov2013}). As a result, the integral representation of the function  $E_{\rho,\mu}(z)$ can be written in the form
\begin{equation*}
  E_{\rho,\mu}(z)=\frac{1}{2\pi i}\int_{\gamma(\varepsilon,\theta_1,\theta_2)}\frac{t^{1/\rho-\mu}e^t}{t^{1/\rho}-z}dt,
\end{equation*}
where the contour $\gamma(\varepsilon,\theta_1,\theta_2)$ is a self-loop consisting of the half-line $S_1=\{t: \arg t=\theta_1, |t|\geqslant\varepsilon\}$, the arc of the circle $C_\varepsilon=\{t: \theta_1\leqslant \arg t\leqslant\theta_2,\ |t|=\varepsilon\}$ and the half-line  $S_2=\{t: \arg t=\theta_2,\ |t|\geqslant\varepsilon\}$. The parameters $\varepsilon, \theta_1, \theta_2$ must satisfy the conditions $\varepsilon>|z|^\rho$, $\theta_1\in (-3\pi/2,-\pi/2)$ and $\theta_2\in (\pi/2,3\pi/2)$ (see theorem~1.1.1 in \cite{Popov2013}).

It is possible to extend the generalization of the integral representation of the Mittag-Leffler function if we take account of some additional properties of the Hankel contour. This work is precisely aimed at obtaining such an integral representation.

\section{ Integral representation of the gamma function}

Let us show that the Hankel contour in the integral representation of the gamma function can be rotated by an arbitrary angle $\psi$ satisfying a condition which will be determined below. At the same time, the value of the gamma function does not depend on the angle $\psi$.  This property allows one to add additional functionality when using the integral representation of the gamma function. In particular, it gives an opportunity to orient the integration contour on the complex plane in the required way, which in some cases turns out to be very useful. This property will be used to obtain the integral representation of the Mittag-Leffler function.

\begin{figure}
  \centering
  \includegraphics[width=0.5\textwidth]{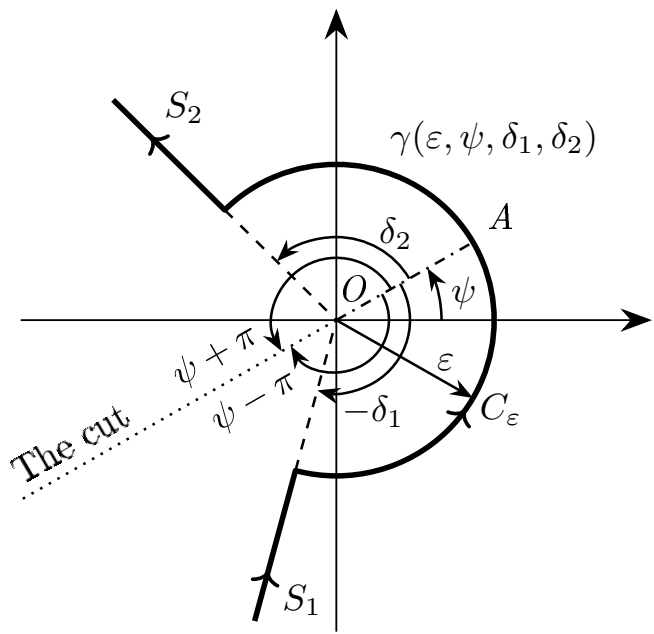}
  \caption{Contour of integration $\gamma(\varepsilon,\psi,\delta_1,\delta_2)$}\label{fig:loop_d1d2psi}
\end{figure}

Consider the contour $\gamma(\varepsilon,\psi,\delta_1,\delta_2)$ consisting of the half-line $S_1$, the arc of the circle $C_\varepsilon$ radius $\varepsilon$ and half-line $S_2$.  Let us assume that the contour is rotated by the angle $\psi$ relative to the origin of coordinates (Fig.~\ref{fig:loop_d1d2psi}). At the same time the rotation of the contour by the angle $\psi$ means that the cut of the complex plane will be rotated by the same angle.  Let the axis that determines the rotation of the contour be an extension of the cut line and makes one straight line with the cut. In Fig.~\ref{fig:loop_d1d2psi} this axis is designated as $OA$. Thus, the lower bank of the cut goes along the half-line $\arg t=\psi-\pi$, and the upper bank of the cut along the half-line $\arg t=\psi+\pi$. Let the angles $\delta_1$ and $\delta_2$ in the common case be not equal to one another and their values lie in the ranges $\pi/2<\delta_1\leqslant\pi$ and $\pi/2<\delta_2\leqslant\pi$. We come to an agreement that these angles are counted from the axis $OA$.
As a result, the following lemma turns out to be true
\begin{lemma}\label{lemma_GF}
For any real $\varepsilon, \delta_1, \delta_2, \psi$, meeting the conditions $\varepsilon>0$, $\pi/2<\delta_1\leqslant\pi$, $\pi/2<\delta_2\leqslant\pi$,
\begin{equation}\label{eq:psiCond}
 \pi/2-\delta_2<\psi<-\pi/2+\delta_1
\end{equation}
any $s\in\C$ the following representation for the gamma function is true
\begin{equation}\label{eq:GammaMy}
  \frac{1}{\Gamma(s)}=\frac{1}{2\pi i}\int_{\gamma(\varepsilon,\psi,\delta_1,\delta_2)} e^tt^{-s}dt,
\end{equation}
where the contour $\gamma(\varepsilon,\psi,\delta_1,\delta_2)$ has the form (see~Fig.~\ref{fig:loop_d1d2psi})
\begin{equation}\label{eq:loop_gamma_d1d2psi}
  \gamma(\varepsilon,\psi,\delta_1,\delta_2)=\left\{\begin{array}{l}
                 S_1=\{t:\ \arg t=-\delta_1+\psi,\ |t|\geqslant\varepsilon\}, \\
                 C_\varepsilon=\{t:\ -\delta_1+\psi\leqslant\arg t\leqslant \delta_2+\psi,|t|=\varepsilon\}, \\
                 S_2=\{t:\ \arg t=\delta_2+\psi,|t|\geqslant\varepsilon\}.
               \end{array}\right.
\end{equation}
\end{lemma}

\begin{proof}
Let us consider the auxiliary integral
\begin{equation*}
  I_1=\int_{C_1} e^tt^{s-1}dt,
\end{equation*}
where the contour $C_1$ consists of the segment $\Gamma_1'$, the arc of the circle $C_R$ radius $R$, the segment $\Gamma_2'$ and the arc of the circle $C_\varepsilon$ radius $\varepsilon$ (Fig.~\ref{fig:loopC1}), defined in the complex plane $t$ in the following way
\begin{equation*}
  C_1=\left\{\begin{array}{l}
               \Gamma_1'=\{t:\ \arg t=\delta_2+\psi,\ \varepsilon\leqslant |t|\leqslant R\}, \\
               C_R=\{t:\ \delta_2+\psi\leqslant\arg t\leqslant\pi,\ |t|=R\}, \\
               \Gamma_2'=\{t:\ \arg t=\pi,\ \varepsilon\leqslant |t|\leqslant R\}, \\
               C_\varepsilon=\{t:\ \delta_2+\psi\leqslant\arg t\leqslant\pi,\ |t|=\varepsilon\}.
             \end{array}\right.
\end{equation*}
where $\pi/2<\delta_2+\psi<3\pi/2$.

\begin{figure}
    \centering
    \includegraphics[width=0.5\textwidth]{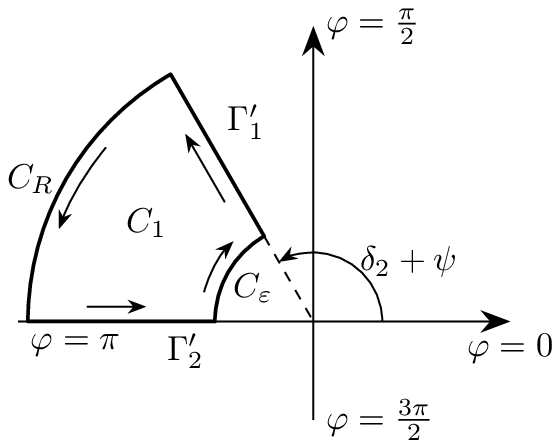}
    \caption{Auxiliary contour $C_1$ in the complex plane $t=re^{i\varphi}$. The values of $\varphi$ change in the limits from 0 to $2\pi$. }\label{fig:loopC1}
\end{figure}

Replacing the variable $t=re^{i\varphi}$  in the integral $I_1$ and directly calculating, we get
\begin{multline}\label{eq:I1'}
  I_1=\int_{\varepsilon}^{R}\exp\left\{r e^{i(\delta_2+\psi)}\right\}r^{s-1} e^{is(\delta_2+\psi)}dr +i\int_{\delta_2+\psi}^{\pi}\exp\left\{Re^{i\varphi}\right\}R^s e^{is\varphi}d\varphi \\
  +\int_{R}^{\varepsilon}\exp\left\{re^{i\pi}\right\}r^{s-1}e^{is\pi}dr
  +i\int_{\pi}^{\delta_2+\psi}\exp\left\{\varepsilon e^{i\varphi}\right\}\varepsilon^s e^{is\varphi}d\varphi\\
  =J_1'+J_2'+J_3'+J_4'.
\end{multline}
In this expression we let $R\to\infty$ and $\varepsilon\to0$. For the integral
\begin{equation}\label{eq:J4'}
  J_4'=i\int_{\pi}^{\delta_2+\psi}\exp\left\{\varepsilon e^{i\varphi}\right\}\varepsilon^s e^{is\varphi}d\varphi
\end{equation}
we have
\begin{equation*}
  \lim_{\varepsilon\to0}|J_4'|\leqslant \lim_{\varepsilon\to0} \int_{\pi}^{\delta_2+\psi}\left|\exp\left\{\varepsilon e^{i\varphi}\right\}\varepsilon^s e^{is\varphi}\right| |d\varphi|.
\end{equation*}
For the integrand we get
\begin{equation}\label{eq:J4'_lim}
\lim_{\varepsilon\to0}\exp\{\varepsilon\cos\varphi+\Re s\ln\varepsilon-\Im s\varphi\}=0,\quad \Re s>0.
\end{equation}
Thus,
\begin{equation}\label{eq:J4'cond}
  \lim_{\varepsilon\to0}|J_4'|=0,\quad\mbox{at}\quad \Re s>0.
\end{equation}
For the integral
\begin{equation}\label{eq:J2'}
  J_2'=i\int_{\delta_2+\psi}^{\pi}\exp\left\{Re^{i\varphi}\right\}R^s e^{is\varphi}d\varphi
\end{equation}
we have
\begin{equation*}
  \lim_{R\to\infty}|J_2'|\leqslant\lim_{R\to\infty}\int_{\delta_2+\psi}^{\pi}\left|\exp\left\{Re^{i\varphi}\right\}R^s e^{is\varphi}\right||d\varphi|.
\end{equation*}
For the integrand we obtain
\begin{equation}\label{eq:limRtoInfty}
  \lim_{R\to\infty}\left|\exp\{Re^{i\varphi}+s\ln R+is\varphi\}\right| =\lim_{R\to\infty}e^{R\cos\varphi}R^{\Re s} e^{\varphi\Im s}
\end{equation}
Since $e^x$ grows faster than any power of the number $x$ then $\lim_{R\to\infty}\exp\{R\cos\varphi+\Re s\ln R-\varphi\Im s \}=0$, if $\cos\varphi<0$. We take interest in the range of values $0\leqslant\varphi\leqslant2\pi$. Thus, the condition $\cos\varphi<0$ leads to the condition $\pi/2<\varphi<3\pi/2$. Here it is worth mentioning that the values $\varphi=\pi/2$ and $\varphi=3\pi/2$ are not included in this interval since at these values $\cos\varphi=0$ and for the limit  (\ref{eq:limRtoInfty}) to be equal to zero, the condition $\Re s<0$ should be met and this condition  contradicts the condition (\ref{eq:J4'cond}).

We substitute here in the inequality $\pi/2<\varphi<3\pi/2$ the lower limit of the integral integration  $J_2'$. As a result, we will get the admissible region of the value for  $\psi$
\begin{equation}\label{eq:I1_psi_cond}
  \pi/2-\delta_2<\psi<3\pi/2-\delta_2.
\end{equation}
Thus, for the integral $J_2'$ we find
\begin{equation}\label{eq:J2'cond}
  \lim_{R\to\infty}|J_2'|=0, \quad\mbox{if}\quad \pi/2-\delta_2<\psi<3\pi/2-\delta_2.
\end{equation}

Now we get back to the integral $I_1$. Consider expression (\ref{eq:I1'}) and assume that  $\varepsilon\to0$ and $R\to\infty$. Taking into consideration (\ref{eq:J4'cond}) and (\ref{eq:J2'cond}), we obtain
\begin{equation}\label{eq:I1'lim_0_infty}
  \lim_{\stackrel{\varepsilon\to0,}{R\to\infty}} I_1=\int_{0}^{\infty}\exp\left\{r e^{i(\delta_2+\psi)}\right\}r^{s-1} e^{is(\delta_2+\psi)}dr
  -\int_{0}^{\infty}\exp\left\{re^{i\pi}\right\}r^{s-1}e^{is\pi}dr,
\end{equation}
where $\Re s>0,\ \pi/2-\delta_2<\psi<3\pi/2-\delta_2$.
Taking into consideration the definition of the gamma function
\begin{equation}\label{eq:GammaFunDef}
  \Gamma(s)=\int_{0}^{\infty}e^{-x}x^{s-1}dx,\quad s\in \C,\ \Re s>0,
\end{equation}
the second summand (\ref{eq:I1'lim_0_infty}) can be written in the form
$\int_{0}^{\infty}\exp\{re^{i\pi}\}r^{s-1}e^{is\pi}dr=e^{is\pi}\Gamma(s)$. Taking into consideration that the contour $C_1$ is a closed one and the integrand does not possess poles inside this contour then according to Cauchy theorem $I_1=\int_{C_1}e^tt^{s-1}dt=0$. As a result, from (\ref{eq:I1'lim_0_infty}) we find
\begin{equation*}
  \int_{0}^{\infty}\exp\left\{r e^{i(\delta_2+\psi)}\right\}r^{s-1} e^{is(\delta_2+\psi)}dr=e^{i\pi s}\Gamma(s),
\end{equation*}
where $\Re s>0,\ \pi/2-\delta_2<\psi<3\pi/2-\delta_2$.
Returning to the complex variable $re^{i\varphi}=t$, this expression can be written in the form
\begin{equation}\label{eq:I1_final}
  \int\limits_{0}^{\infty\cdot\exp\{i(\delta_2+\psi)\}} e^tt^{s-1}dt=e^{i\pi s}\Gamma(s),\quad \Re s>0,\quad \pi/2-\delta_2<\psi<3\pi/2-\delta_2.
\end{equation}

Next, we consider the second auxiliary integral
\begin{equation*}
  I_2=\int_{C_2} e^tt^{s-1}dt,
\end{equation*}
where the contour $C_2$ consists of the segment $\Gamma_1''$, the arc of the circle $C_\varepsilon''$ radius $\varepsilon$, the segment $\Gamma_2''$ and arc of the circle $C_R''$ radius $R$ (Fig.~\ref{fig:loopC2}) defined in the complex plane  $t$ in the following way
\begin{equation*}
  C_2=\left\{\begin{array}{l}
               \Gamma_1''=\{t:\ \arg t=-\delta_1+\psi,\ \varepsilon\leqslant |t|\leqslant R\}, \\
               C_\varepsilon''=\{t:\ -\delta_1+\psi\leqslant\arg t\leqslant -\pi,\ |t|=\varepsilon\}, \\
               \Gamma_2''=\{t:\ \arg t=-\pi,\ \varepsilon\leqslant |t|\leqslant R\}, \\
               C_R''=\{t:\ \delta_2+\psi\leqslant\arg t\leqslant\pi,\ |t|=R\}.
             \end{array}\right.
\end{equation*}
where $-3\pi/2<-\delta_1+\psi<-\pi/2$.

\begin{figure}
    \centering
    \includegraphics[width=0.5\textwidth]{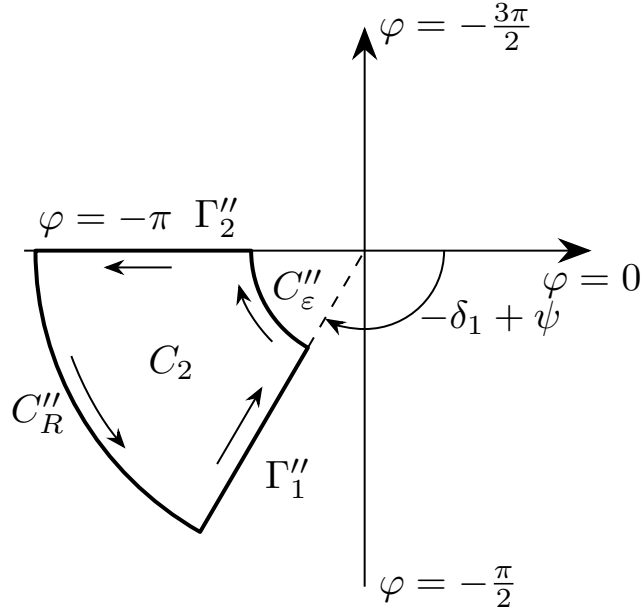}
    \caption{Auxiliary contour $C_2$ in the complex plane $t=re^{i\varphi}$. The values $\varphi$ change within the limits from 0 to $-2\pi$. }\label{fig:loopC2}
\end{figure}

Similarly to the previous case, we get
\begin{multline}\label{eq:I2'_sum}
  I_2=\int_{\varepsilon}^{R}\exp\left\{re^{-i\pi}\right\}r^{s-1}e^{-is\pi}dr
  +i\int_{-\pi}^{-\delta_1+\psi}\exp\left\{R e^{i\varphi}\right\}R^s e^{is\varphi}d\varphi\\
  +\int_{R}^{\varepsilon}\exp\left\{r e^{i(-\delta_1+\psi)}\right\}r^{s-1} e^{is(-\delta_1+\psi)}dr +i\int_{-\delta_1+\psi}^{-\pi}\exp\left\{\varepsilon e^{i\varphi}\right\}\varepsilon^s e^{is\varphi}d\varphi\\
  =J_1''+J_2''+J_3''+J_4''.
\end{multline}

Passing in the integral
\begin{equation*}
  J_4''=i\int_{-\delta_1+\psi}^{-\pi}\exp\left\{\varepsilon e^{i\varphi}\right\}\varepsilon^s e^{is\varphi}d\varphi
\end{equation*}
to the limit $\varepsilon\to0$, one can show that
\begin{equation}\label{eq:J4''cond}
  \lim_{\varepsilon\to0}|J_4''|=0,\quad\Re s>0.
\end{equation}

Next, passing in the integral
\begin{equation}\label{eq:J2''}
J_2''= i\int_{-\pi}^{-\delta_1+\psi}\exp\left\{R e^{i\varphi}\right\}R^s e^{is\varphi}d\varphi
\end{equation}
to the limit $R\to\infty$, one can show that
\begin{equation}\label{eq:J2''cond}
  \lim_{R\to\infty}|J_2''|=0,
\end{equation}
if the condition is met
\begin{equation}\label{eq:I2_psi_cond}
  -3\pi/2+\delta_1<\psi<-\pi/2+\delta_1.
\end{equation}

Now we get back to the expression (\ref{eq:I2'_sum}).  Assuming in this expression that $\varepsilon\to0$, $R\to\infty$ and taking into consideration (\ref{eq:J4''cond}) and (\ref{eq:J2''cond}) we obtain
\begin{equation}\label{eq:I2'_lim_0_infty}
  \lim_{\stackrel{\varepsilon\to0}{R\to\infty}}I_2=\int_{0}^{\infty}\exp\left\{re^{-i\pi}\right\}r^{s-1}e^{-is\pi}dr-
  \int_{0}^{\infty}\exp\left\{r e^{i(-\delta_1+\psi)}\right\}r^{s-1} e^{is(-\delta_1+\psi)}dr.
\end{equation}
where $\Re s>0$, $-3\pi/2+\delta_1<\psi<-\pi/2+\delta_1$. Using (\ref{eq:GammaFunDef}), the first summand can be represented in the form
$ \int_{0}^{\infty}\exp\left\{re^{-i\pi}\right\}r^{s-1}e^{-i\pi s}dr=e^{-i\pi s}\Gamma(s)$. Taking into consideration that the contour of integration $C_2$ is a closed contour and inside this contour there are no singular points, then according to Cauchy theorem $I_2=\int_{C_2}e^tt^{s-1}dt=0$. From (\ref{eq:I2'_lim_0_infty}) we find
\begin{equation*}
  \int_{0}^{\infty}\exp\left\{r e^{i(-\delta_1+\psi)}\right\}r^{s-1} e^{is(-\delta_1+\psi)}dr=e^{-i\pi s}\Gamma(s),
\end{equation*}
where $\Re s>0$, $-3\pi/2+\delta_1<\psi<-\pi/2+\delta_1$.
Passing in the expression to the complex variable $re^{i\varphi}=t$, we derive
\begin{equation}\label{eq:I2_final}
  \int\limits_{0}^{\infty\cdot\exp\{i(-\delta_1+\psi)\}}e^tt^{s-1}dt=e^{-i\pi s}\Gamma(s),\quad\Re s>0,\quad -3\pi/2+\delta_1<\psi<-\pi/2+\delta_1.
\end{equation}

Now we consider the auxiliary integral
\begin{equation}\label{eq:I}
  I(s)=\int\limits_{\gamma(\varepsilon, \psi, \delta_1,\delta_2)}e^tt^{s-1}dt,
\end{equation}
where the contour $\gamma(\varepsilon,\psi,\delta_1,\delta_2)$  is defined by the expression (\ref{eq:loop_gamma_d1d2psi}). By replacing in this integral the integration variable $t=re^{i\varphi}$ and calculating the integral directly we find
\begin{multline}\label{eq_Iexpand}
  I(s)=\int\limits_{\infty}^{\varepsilon}\exp\left\{re^{i(-\delta_1+\psi)}\right\}r^{s-1}e^{is(-\delta_1+\psi)}dr\\
  +i\int\limits_{-\delta_1+\psi}^{\delta+\psi}\exp\left\{\varepsilon e^{i\varphi}\right\}\varepsilon^s e^{is\varphi}d\varphi
  +\int\limits_{\varepsilon}^{\infty}\exp\left\{re^{i(\delta_2+\psi)}\right\}r^{s-1}e^{is(\delta_2+\psi)}dr.
\end{multline}

Now we consider the second summand in the right part of this expression and we pass in this summand to the limit $\varepsilon\to0$. As a result, we have
\begin{equation*}
  \lim_{\varepsilon\to0}\left|i\int\limits_{-\delta_1+\psi}^{\delta+\psi}\exp\left\{\varepsilon e^{i\varphi}\right\}\varepsilon^s e^{is\varphi}d\varphi\right|\leqslant
  \lim_{\varepsilon\to0} \int\limits_{-\delta_1+\psi}^{\delta+\psi}\left|\exp\left\{\varepsilon e^{i\varphi}\right\}\varepsilon^s e^{is\varphi}\right| |d\varphi|
\end{equation*}
Calculating the limit of the integrand, it is possible to show that
$$
\lim_{\varepsilon\to0}\left|\exp\left\{\varepsilon e^{i\varphi} + s\ln\varepsilon +is\varphi\right\}\right|=0,\quad\mbox{at}\quad\Re s>0.
$$
Thus, we obtain
\begin{equation}\label{eq:J2_lim}
  \lim_{\varepsilon\to0}\left|i\int_{-\delta_1+\psi}^{\delta+\psi}\exp\left\{\varepsilon e^{i\varphi}\right\}\varepsilon^s e^{is\varphi}d\varphi\right|=0,\quad\mbox{at}\quad\Re s>0.
\end{equation}

We return to (\ref{eq_Iexpand}) and assume that $\varepsilon\to0$. Taking into consideration (\ref{eq:J2_lim}) we obtain
\begin{multline}\label{eq:I_sum}
  \lim_{\varepsilon\to0}I(s)=\int\limits_{0}^{\infty}\exp\left\{re^{i(\delta_2+\psi)}\right\}r^{s-1}e^{is(\delta_2+\psi)}dr
  -\int\limits_{0}^{\infty}\exp\left\{re^{i(-\delta_1+\psi)}\right\}r^{s-1}e^{is(-\delta_1+\psi)}dr \\
  =\int\limits_{0}^{\infty\cdot \exp\{i(\delta_2+\psi)\}} e^t t^{s-1}dt-\int\limits_{0}^{\infty\cdot \exp\{i(-\delta_1+\psi)\}} e^t t^{s-1}dt,
\end{multline}
where we again passed to the complex variable $t$. Substituting now in (\ref{eq:I_sum}) the expressions (\ref{eq:I1_final}) and (\ref{eq:I2_final}), we get
\begin{equation*}
  I(s)=e^{is\pi}\Gamma(s)-e^{-is\pi}\Gamma(s)=2i\sin(\pi s)\Gamma(s)=\frac{2\pi i}{\Gamma(1-s)},
\end{equation*}
where the property $\Gamma(s)\Gamma(1-s)=\frac{\pi}{\sin(\pi s)}$ was used. The intersection of  the regions (\ref{eq:I1_psi_cond}) and (\ref{eq:I2_psi_cond}) gives the range of admissible values of the angle $\psi$, determined by inequality (\ref{eq:psiCond}).

Returning to (\ref{eq:I}) we get
\begin{equation}\label{eq:I_final}
  \int_{\gamma(\varepsilon,\psi,\delta_1,\delta_2)} e^tt^{s-1}dt=\frac{2\pi i}{\Gamma(1-s)},\quad \Re s>0,
\end{equation}
where the contour of integration $\gamma(\varepsilon,\psi, \delta_1,\delta_2)$ is defined by (\ref{eq:loop_gamma_d1d2psi}). As we can see, both sides of this equality are entire functions and coincide for $\Re s>0$. Therefore, by the uniqueness theorem, they will coincide on the whole complex plane. By making in (\ref{eq:I_final}) a substitution of  $1-s$ for $s$, we get
\begin{equation*}
  \frac{1}{\Gamma(s)}=\frac{1}{2\pi i}\int_{\gamma(\varepsilon,\psi,\delta_1,\delta_2)} e^tt^{-s}dt,
\end{equation*}
for any values $s\in\C$.
\begin{flushright}
  $\Box$
\end{flushright}
\end{proof}

The proved lemma shows that the contour of integration $\gamma(\varepsilon,\psi,\delta_1,\delta_2)$ in the representation (\ref{eq:GammaMy}) can be rotated by an arbitrary angle $\psi$ relative to the origin of coordinates. In Fig.~\ref{fig:loop_gammma_rot} the region of rotation is grey. The value of the gamma function does not depend on the parameters $\varepsilon,\psi,\delta_1,\delta_2$. The contour rotation angle can lie in the interval $\psi\in(\tfrac{\pi}{2}-\delta_2, -\tfrac{\pi}{2}+\delta_1)$. However, the extreme values of this interval the angle $\psi$ cannot be taken, since in this case the integral on the right-hand side (\ref{eq:GammaMy}) will diverge. Indeed, let $\psi=\pi/2-\delta_2$. In this case, the half-line $S_2$ will go along the positive part of the imaginary axis. This means that in the integral $J_2'$ (see~(\ref{eq:J2'})) the lower limit of integration  $\delta_2+\psi=\pi/2$ and the expression (\ref{eq:limRtoInfty}) at the point of the lower limit, will take the form
\begin{equation*}
    \lim_{R\to\infty}\left|\exp\{Re^{i\pi/2}+s\ln R+is\tfrac{\pi}{2}\}\right| =\lim_{R\to\infty}\exp\{\Re s\ln R+\tfrac{\pi}{2}\Im s\} =\infty
\end{equation*}
Thus, at $\psi=\pi/2-\delta_2$ the integral $J_2'$ diverges, which in its turn, leads to a divergence of the integral $I_1$ (see~(\ref{eq:I1'lim_0_infty})) and, consequently, to a divergence of  (\ref{eq:GammaMy}). If we assume that $\Re s<0$, then the this limit will be equal to zero but in this case the integral (\ref{eq:J4'}) will diverge, since the limit (\ref{eq:J4'_lim}) at $\Re s<0$ is equal to the infinity. We have a similar situation for the other extreme value $\psi=-\pi/2+\delta_1$. In this case the half-line $S_1$ goes along the negative part of the imaginary axis which leads to the integral divergence $J_2''$ (see~(\ref{eq:J2''})).

\begin{figure}
  \centering
  \includegraphics[width=0.5\textwidth]{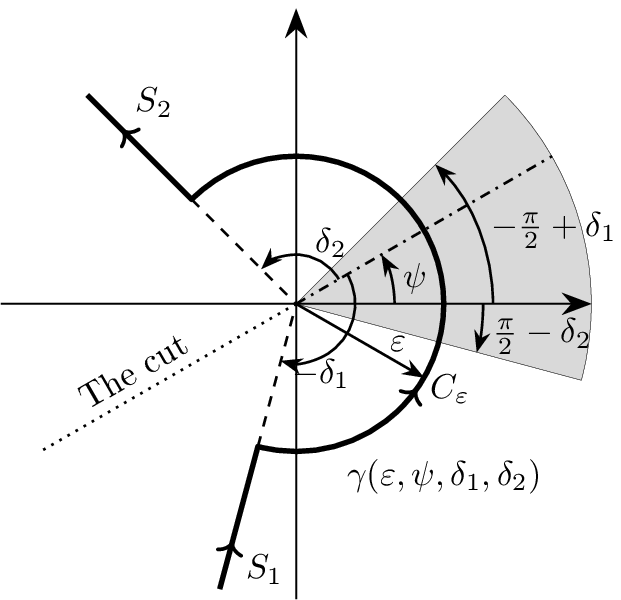}
  \caption{Contour of integration $\gamma(\varepsilon,\psi,\delta_1,\delta_2)$. The range of angle values satisfying the condition (\ref{eq:psiCond}) is shaded grey }\label{fig:loop_gammma_rot}
\end{figure}

It follows from lemma~\ref{lemma_GF} that the contour of integration $\gamma(\varepsilon,\psi,\delta_1,\delta_2)$ can be rotated within the limits of the angular sector determined by the condition (\ref{eq:psiCond}). This rotation is the rotation of the contour in the complex plane $t$ and it is not the rotation of the complex plane $t$. The limit given by the condition (\ref{eq:psiCond}) does not give an opportunity to use this formula for other values of the angles $\psi$. This circumstance limits the utility of lemma~\ref{lemma_GF}. The following corollary gives an opportunity to expand the range of admissible values of the parameter $\psi$.

\begin{corollary}\label{coroll:GF_lambda}
For any $s,\lambda\in\C$, where $\lambda\neq0$, any real $\varepsilon$, $\delta_1$, $\delta_2$, $\psi_\lambda$, that $\varepsilon>0$, $\pi/2<\delta_1\leqslant\pi$, $\pi/2<\delta_2\leqslant\pi$,
\begin{equation}\label{eq:psiLambdaCond}
\pi/2-\delta_2-\arg\lambda<\psi_\lambda<-\pi/2+\delta_1-\arg\lambda
\end{equation}
the following representation for the gamma function is valid
\begin{equation}\label{eq:GF_lambda}
  \frac{1}{\Gamma(s)}=\frac{\lambda^{1-s}}{2\pi i}\int_{\gamma_\lambda(\varepsilon,\psi_\lambda,\delta_1,\delta_2)} e^{\lambda \tau}\tau^{-s}d\tau,
\end{equation}
where the contour $\gamma_\lambda(\varepsilon,\psi_\lambda, \delta_1,\delta_2)$ has the form
\begin{equation}\label{eq:loop_gamma_psiLambda}
  \gamma_\lambda(\varepsilon,\psi_\lambda, \delta_1,\delta_2)=\left\{\begin{array}{l}
                 S_1^\lambda=\{\tau:\ \arg \tau=-\delta_1+\psi_\lambda,\ |\tau|\geqslant\varepsilon/|\lambda|\}, \\
                 C_\varepsilon^\lambda=\{\tau:\ -\delta_1+\psi_\lambda\leqslant\arg \tau\leqslant \delta_2+\psi_\lambda,|\tau|=\varepsilon/|\lambda|\}, \\
                 S_2^\lambda=\{\tau:\ \arg \tau=\delta_2+\psi_\lambda,|t|\geqslant\varepsilon/|\lambda|\}.
               \end{array}\right.
\end{equation}
\end{corollary}

\begin{proof}
We will make in (\ref{eq:GammaMy}) a substitution for the integration variable $t=\lambda\tau$, where $\lambda\in\C$ we have
\begin{equation*}
  \frac{1}{\Gamma(s)}=\frac{\lambda^{1-s}}{2\pi i}\int_{\gamma_\lambda
  (\varepsilon,\psi_\lambda, \delta_1,\delta_2)}e^{\lambda\tau}\tau^{-s}d\tau,
\end{equation*}
Now we consider how the contour of integration is transformed  $\gamma(\varepsilon,\psi,\delta_1,\delta_2)$. From the equality $t=\lambda\tau$ it follows \begin{equation*}
  \tau=\frac{t}{\lambda}=\frac{|t|}{|\lambda|}\exp\{i(\arg t-\arg\lambda)\}\ \Rightarrow\
  |\tau|=\frac{|t|}{|\lambda|},\ \arg\tau=\arg t-\arg\lambda
\end{equation*}
From here we obtain that the half-line $S_1$ of the contour $\gamma(\varepsilon,\psi,\delta_1,\delta_2)$ maps into the half-line $S_1^\lambda=\{\tau:\ \arg\tau=-\delta_1+\psi-\arg\lambda,\ |\tau|\geqslant \varepsilon   /|\lambda|\}$, the arc of the circle  $C_\varepsilon$ maps into the arc of the circle $C_\varepsilon^\lambda=\{\tau:\ \delta_1+\psi-\arg\lambda\leqslant\arg\tau\leqslant\delta_2+\psi-\arg\lambda,\ |\tau|=\varepsilon/|\lambda|\}$, and the half-line $S_2$ maps into the half-line $S_2^\lambda=\{\tau:\ \arg\tau=\delta_2+\psi-\arg\lambda,\ |\tau|\geqslant \varepsilon/|\lambda|\}$. The condition (\ref{eq:psiCond}) takes the form
\begin{equation*}
  \pi/2-\delta_2-\arg\lambda<\psi-\arg\lambda<-\pi/2+\delta_1-\arg\lambda.
\end{equation*}

We will introduce the notation $\psi_\lambda=\psi-\arg\lambda$. As a result, we find that the contour  $\gamma(\varepsilon,\psi,\delta_1,\delta_2)$ maps into the contour
\begin{equation*}
  \gamma_\lambda(\varepsilon, \psi_\lambda, \delta_1,\delta_2)=\left\{\begin{array}{l}
                 S_1^\lambda=\{\tau:\ \arg \tau=-\delta_1+\psi_\lambda,\ |t|\geqslant\varepsilon/|\lambda|\}, \\
                 C_\varepsilon^\lambda=\{\tau:\ -\delta_1+\psi_\lambda\leqslant\arg \tau\leqslant \delta_2+\psi_\lambda,|\tau|=\varepsilon/|\lambda|\}, \\
                 S_2^\lambda=\{\tau:\ \arg t=\delta_2+\psi_\lambda,|t|\geqslant\varepsilon/|\lambda|\}.
               \end{array}\right.
\end{equation*}
where $\pi/2-\delta_2-\arg\lambda<\psi_\lambda<-\pi/2+\delta_1-\arg\lambda.$
\begin{flushright}
  $\Box$
\end{flushright}
\end{proof}

As we can see, the substitution $t=\lambda\tau$ is the conform mapping of a complex plane $t$ into the plane $\tau$, which is the rotation of the complex plane $t$ by the angle $\arg\lambda$ and stretching this plane by the value $|\lambda|$. If $\arg\lambda>0$ the plane $t$ is rotated clockwise and at $\arg\lambda<0$ -- counterclockwise, at $|\lambda|>1$ the compression of the plane $t$ takes place, and at $|\lambda|<1$ -- stretching of the plane $t$. Here the only one restraint $\lambda\neq0$ is imposed on values $\lambda$. Without loss of generality, we can choose $|\lambda|=1$. In this case, the mapping will be the rotation of the plane as a whole by an angle $\arg\lambda$. As a result, \emph{the rotation of the integration contour of the gamma function can be represented as a sum of two independent rotations: the rotation of the complex plane as a whole by an angle $\arg\lambda$ and rotation of the integration contour on the complex plane by an angle $\psi_\lambda$.} The value of the angle $\psi_\lambda$ can be chosen arbitrarily only if it could satisfy the condition (\ref{eq:psiLambdaCond}). At the same time, rotating by an angle $\psi_\lambda$ is precisely the rotation of the contour of integration on the complex plane and does not lead to any transformation of the complex plane itself. Thus, choosing in a certain way the value $\arg\lambda$ one can map the angular sector of the plane $t$, determined by the condition (\ref{eq:psiCond}), into the required range of values $\arg\tau$  of the plane $\tau$.

\section{Integral representation of the Mittag-Leffler function}

Let us now return to the Mittag-Leffler function and get an integral representation for this function. When deriving the integral representation, we will use corollary~\ref{coroll:GF_lambda}. Using this corollary turns out to be useful and allows getting  a more general form of the integral representation of the Mittag-Leffler function. As a result, the following theorem is true.

\begin{theorem}\label{lemm:MLF_int}
For any real $\rho, \delta_{1\rho}, \delta_{2\rho}, \epsilon$ that $\rho>1/2$, $\frac{\pi}{2\rho}<\delta_{1\rho}\leqslant\min(\pi,\pi/\rho)$, $\frac{\pi}{2\rho}<\delta_{2\rho}\leqslant\min(\pi,\pi/\rho)$, $\epsilon>0$, any $\mu\in\C$ and  any $z\in\C$ that
\begin{equation}\label{eq:z_cond_lemm}
  \frac{\pi}{2\rho}-\delta_{2\rho}+\pi<\arg z<-\frac{\pi}{2\rho}+\delta_{1\rho}+\pi
\end{equation}
the Mittag-Leffler function can be represented in the form
\begin{equation}\label{eq:MLF_int}
  E_{\rho,\mu}(z)=\frac{\rho}{2\pi i} \int_{\gamma_\zeta(\epsilon,\arg z,\delta_{1\rho},\delta_{2\rho})}\frac{\exp\left\{(z\zeta)^{\rho}\right\}(z\zeta)^{\rho(1-\mu)}}{\zeta-1}d\zeta.
\end{equation}
where the contour of integration $\gamma_\zeta$ has the form (Fig.~\ref{fig:loop_gammaZeta})
\begin{equation}\label{eq:loop_gammaZeta}
  \gamma_\zeta(\epsilon,\arg z,\delta_{1\rho},\delta_{2\rho})=\left\{\begin{array}{ll}
                       S_1=&\{\zeta: \arg\zeta=-\delta_{1\rho}-\pi,\quad |\zeta|\geqslant 1+\epsilon\},\\
                       C_\epsilon=&\{\zeta: -\delta_{1\rho}-\pi\leqslant\arg\zeta\leqslant\delta_{2\rho}- \pi,\quad |\zeta|=1+\epsilon\},\\
                       S_2=&\{\zeta: \arg\zeta=\delta_{2\rho}-\pi,\quad|\zeta|\geqslant1+\epsilon\}.
                     \end{array}\right.
\end{equation}
\end{theorem}

\begin{figure}
  \centering
  \includegraphics[width=0.4\textwidth]{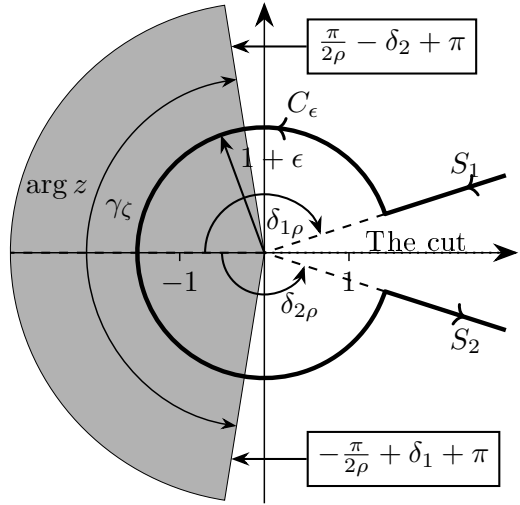}
  \caption{The contour of integration $\gamma_\zeta$. The region corresponding to the condition (\ref{eq:z_cond_lemm}) is shaded grey}\label{fig:loop_gammaZeta}
\end{figure}

\begin{proof}
We will use Corollary~\ref{coroll:GF_lambda}. By making a substitution in (\ref{eq:GF_lambda}) of the integration variable $\tau=u^\rho$, we consider how the contour of integration is $\gamma_\lambda(\varepsilon,\psi_\lambda,\delta_1,\delta_2)$ is transformed. As a result, we have
\begin{equation}\label{eq:trasform_u^rho}
  u=\tau^{1/\rho}=|\tau|^{1/\rho}\exp\{i (1/\rho) \arg\tau\}\ \Rightarrow\ |u|=|\tau|^{1/\rho},\ \arg u=(1/\rho)\arg\tau.
\end{equation}
In the complex plane $\tau$ the range of admissible values for the angle $\delta_{1}$ is determined by the angular sector $\pi/2<\delta_1\leqslant\pi$  with the span angle $\pi/2$. Using (\ref{eq:trasform_u^rho}) we obtain that the angular sector is mapped into the angular sector $\pi/(2\rho)<\delta_{1\rho}\leqslant\pi/\rho$ in the plane $u$ with the span angle $\pi/(2\rho)$,  where $\delta_{1\rho}=\delta_1/\rho$. As we can see at $1/2<\rho<1$ the span angle of this sector turns out to be more than $\pi/2$. That is why, we limit the range of admissible values of the angle  $\delta_{1\rho}$ in such a way that the right-hand boundary of this sector should not exceed $\pi$. Thus, at $\rho<1$ we have $\pi/(2\rho)<\delta_{1\rho}\leqslant\pi$. At $\rho\geqslant1$ the span angle of the sector is less than $\pi/2$ and therefore the right-hand boundary of this sector for any  $\rho\geqslant1$ will be less than $\pi$. Joining these two cases we get
\begin{equation*}
  \pi/(2\rho)<\delta_{1\rho}\leqslant\min(\pi,\pi/\rho).
\end{equation*}
Similarly, for the angle  $\delta_2$ the condition $\pi/2<\delta_2\leqslant\pi$ is transformed into the condition $\pi/(2\rho)<\delta_{2\rho}\leqslant\min(\pi,\pi/\rho)$. It should be pointed out that if we do not limit the maximal values of the angles $\delta_{1\rho}$ and $\delta_{2\rho}$ by the value of $\pi$, then the sum $\delta_{1\rho}+\delta_{2\rho}$ may be more than $2\pi$. Thus, the half-lines $S_1$ or $S_2$, of the integration contour can intersect the cut of the complex plane and go to another sheet of the Riemann surface. Hence, the meaning of this restriction becomes clear:  with such a constraint on the values of $\delta_{1\rho}$ and $\delta_{2\rho}$, we always remain on the same sheet of the Riemann surface.

Using (\ref{eq:trasform_u^rho}) we obtain that the half-line $S_1^\lambda$ of the contour $\gamma_\lambda(\varepsilon,\psi_\lambda,\delta_1,\delta_2)$ (see~(\ref{eq:loop_gamma_psiLambda})) maps into the half-line $S_{1\rho}=\{u:\ \arg u= -\delta_{1\rho}+\psi_{\rho\lambda},\ |u|\geqslant\varepsilon^{1/\rho}/|\lambda|\}$, where $\psi_{\rho\lambda}=\psi_\lambda/\rho$. The arc of the circle $C_\varepsilon^\lambda$ maps into the arc of the circle $C_{\varepsilon\rho}=\{u:\ -\delta_{1\rho}+\psi_{\rho\lambda}\leqslant\arg u\leqslant\delta_{2\rho}+\psi_{\rho\lambda},\  |u|=\varepsilon^{1/\rho}/|\lambda|\}$, and the half-line $S_2^\lambda$ into the half-line   $S_{2\rho}=\{u:\ \arg u= \delta_{2\rho}+\psi_{\rho\lambda},\ |u|\geqslant\varepsilon^{1/\rho}/|\lambda|\}$. Thus, the contour of integration $\gamma_\lambda(\varepsilon,\psi_\lambda,\delta_1,\delta_2)$ maps into the contour
\begin{equation*}
  \gamma_\rho(\varepsilon,\psi_{\rho\lambda},\delta_{1\rho},\delta_{2\rho})=\left\{\begin{array}{l}
                 S_{1\rho}=\{u:\ \arg u=-\delta_{1\rho}+\psi_{\rho\lambda},\ |u|\geqslant\varepsilon^{1/\rho
                 }/|\lambda|\}, \\
                 C_{\varepsilon\rho}=\{u:\ -\delta_{1\rho}+\psi_{\rho\lambda}\leqslant\arg u\leqslant\delta_{2\rho}+\psi_{\rho\lambda},\  |u|=\varepsilon^{1/\rho}/|\lambda|\}, \\
                 S_{2\rho}=\{u:\ \arg u= \delta_{2\rho}+\psi_{\rho\lambda},\ |u|\geqslant\varepsilon^{1/\rho}/|\lambda|\},
               \end{array}\right.
\end{equation*}
where $\pi/(2\rho)<\delta_{1\rho}\leqslant\min(\pi,\pi/\rho)$, $\pi/(2\rho)<\delta_{2\rho}\leqslant\min(\pi,\pi/\rho)$,
    $\tfrac{\pi}{2\rho}-\delta_{2\rho}-\tfrac{1}{\rho}\arg\lambda< \psi_{\rho\lambda}< -\tfrac{\pi}{2\rho}+\delta_{1\rho}-\tfrac{1}{\rho}\arg\lambda$.

Now making a substitution in the integral (\ref{eq:GF_lambda}) of the integration variable $t=u^\rho$, we obtain
\begin{equation*}
  \frac{1}{\Gamma(s)}=\frac{\rho\lambda^{1-s}}{2\pi i}\int_{\gamma_\rho(\varepsilon,\psi_{\rho\lambda},\delta_{1\rho},\delta_{2\rho})}e^{\lambda u^\rho} u^{\rho(1-s)-1}d u.
\end{equation*}

By substituting this expression in the Mittag-Leffler (\ref{eq:MLF_gen}), we get
\begin{multline}\label{eq:MLF_tmp1}
  E_{\rho,\mu}(z)=\sum_{k=0}^{\infty}\frac{z^k}{\Gamma(\mu+k/\rho)}=\sum_{k=0}^{\infty}z^k \frac{\rho\lambda^{1-\mu-\frac{k}{\rho}}}{2\pi i}\int_{\gamma_\rho(\varepsilon,\psi_{\rho\lambda},\delta_{1\rho},\delta_{2\rho})}e^{\lambda u^\rho} u^{\rho(1-\mu-\frac{k}{\rho})-1} du \\
  =\frac{\rho\lambda^{\frac{1}{\rho}}}{2\pi i}\int_{\gamma_\rho(\varepsilon,\psi_{\rho\lambda},\delta_{1\rho},\delta_{2\rho})}e^{\left(\lambda^{1/\rho} u\right)^\rho} \left(\lambda^{1/\rho} u\right)^{\rho(1-\mu)-1}\sum_{k=0}^{\infty}\left(\frac{z}{\lambda^{1/\rho} u}\right)^k d u.
\end{multline}

Since $\varepsilon$ is arbitrary, we choose it so that $\varepsilon>|z|/|\lambda|^{1/\rho}$. Hence it follows that
\begin{equation}\label{eq:uz_cond}
  \sup_{u\in \gamma_\rho(\varepsilon,\psi_{\rho\lambda},\delta_{1\rho},\delta_{2\rho})}\left|\frac{z}{\lambda^{1/\rho} u}\right|<1.
\end{equation}
Using the geometric progression formula for the sum under the integral sign in (\ref{eq:MLF_tmp1}) we have
\begin{equation*}
  \sum_{k=0}^{\infty}\left(\frac{z}{\lambda^{1/\rho}u} \right)^k = \lim_{n\to\infty}\sum_{k=0}^{n}\left(\frac{z}{\lambda^{1/\rho}u}\right)^k = \lim_{n\to\infty}\frac{1-\left(\frac{z}{\lambda^{1/\rho}u}\right)^n}{1-\frac{z}{\lambda^{1/\rho}u}}= \frac{1}{1-\frac{z}{\lambda^{1/\rho}u}}.
\end{equation*}
Substituting now this result into (\ref{eq:MLF_tmp1}), we obtain
\begin{equation}\label{eq:MLF_tmp2}
  E_{\rho,\mu}(z)=\frac{\rho\lambda^{1/\rho}}{2\pi i}\int_{\gamma_\rho(\varepsilon,\psi_{\rho\lambda},\delta_{1\rho},\delta_{2\rho})} \frac{\exp\left\{\left(\lambda^{1/\rho}u\right)^\rho\right\}\left(\lambda^{1/\rho}u\right)^{\rho(1-\mu)-1}} {1-\frac{z}{\lambda^{1/\rho}u}}d u.
\end{equation}

We restrict ourselves here to considering the case
\begin{equation}\label{eq:argZ_cond}
  \pi/2<\arg z<3\pi/2
\end{equation}
and will choose $\arg\lambda=-\rho\pi$. This choice of the value $\arg\lambda$ means the rotation of the contour $\gamma_\rho(\varepsilon,\psi_{\rho\lambda},\delta_{1\rho},\delta_{2\rho})$ by the angle $-\pi$ counterclockwise. Indeed, as we can see from (\ref{eq:trasform_u^rho}) due to the substitution  of a variable $t=u^\rho$ the rotation angle corresponding to $\arg\lambda$ in the plane $t$ maps on the plane  $u$ into the angle $\tfrac{1}{\rho}\arg\lambda$. As a result, to rotate the contour in the plane $u$ by the angle $-\pi$, it is necessary to choose $\arg\lambda=-\rho\pi$.  Thus, the contour $\gamma_\rho(\varepsilon,\psi_{\rho\lambda},\delta_{1\rho},\delta_{2\rho})$ takes the form (see~Fig.~\ref{fig:loop_gammaRho_rhoPi})
\begin{equation*}
  \gamma_\rho(\varepsilon,\psi_{\rho\lambda},\delta_{1\rho},\delta_{2\rho})=\left\{\begin{array}{l}
                 S_{1\rho}=\{u:\ \arg u=-\delta_{1\rho}+\psi_{\rho\lambda},\ |u|\geqslant\varepsilon^{1/\rho
                 }/|\lambda|\}, \\
                 C_{\varepsilon\rho}=\{u:\ -\delta_{1\rho}+\psi_{\rho\lambda}\leqslant\arg u\leqslant\delta_{2\rho}+\psi_{\rho\lambda},\  |u|=\varepsilon^{1/\rho}/|\lambda|\}, \\
                 S_{2\rho}=\{u:\ \arg u= \delta_{2\rho}+\psi_{\rho\lambda},\ |u|\geqslant\varepsilon^{1/\rho}/|\lambda|\},
               \end{array}\right.
\end{equation*}
where $\pi/(2\rho)<\delta_{1\rho}\leqslant\min(\pi,\pi/\rho)$, $\pi/(2\rho)<\delta_{2\rho}\leqslant\min(\pi,\pi/\rho)$ and
\begin{equation}\label{eq:psi_rhoPi_cond}
\tfrac{\pi}{2\rho}-\delta_{2\rho}+\pi< \psi_{\rho\lambda}< -\tfrac{\pi}{2\rho}+\delta_{1\rho}+\pi.
\end{equation}

\begin{figure}
  \centering
  \includegraphics[width=0.49\textwidth]{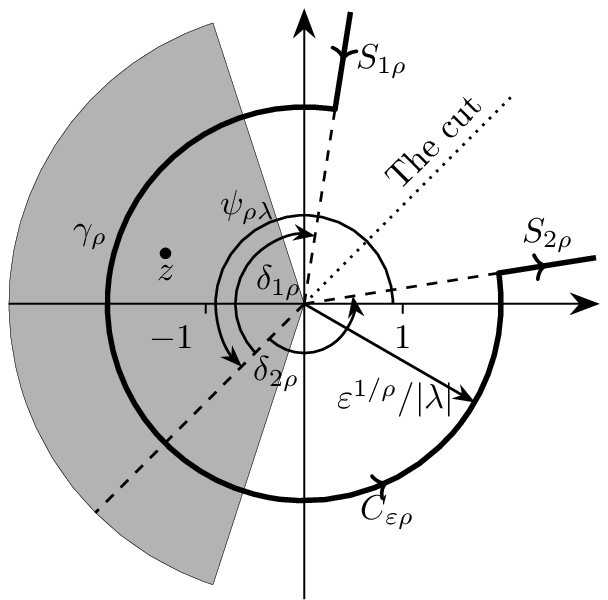}\hfill
  \includegraphics[width=0.49\textwidth]{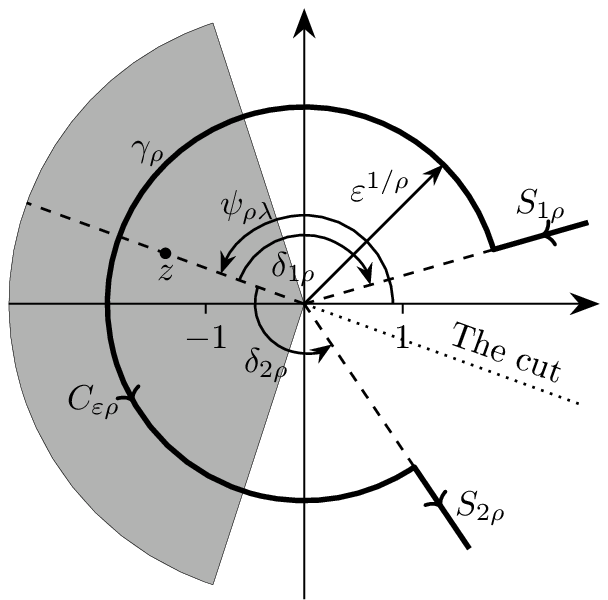}
  \parbox[t]{0.475\textwidth}{\caption{Contour $\gamma_\rho(\varepsilon,\psi_{\rho\lambda},\delta_{1\rho},\delta_{2\rho})$ at $\arg\lambda=-\rho\pi$. The angular region corresponding to the condition (\ref{eq:psi_rhoPi_cond}) is shaded grey }\label{fig:loop_gammaRho_rhoPi}}\hfill
  \parbox[t]{0.475\textwidth}{\caption{Contour $\gamma_\rho(\varepsilon,\psi_{\rho\lambda},\delta_{1\rho},\delta_{2\rho})$ at $\arg\lambda=-\rho\pi$ and $\psi_{\rho\lambda}=\arg z$. The angular region corresponding to the condition (\ref{eq:z_cond}) is grey}\label{fig:loop_gammaRho_rhoPi_argZ}}
\end{figure}

As a result, we have rotated the contour $\gamma_\rho(\varepsilon,\psi_{\rho\lambda},\delta_{1\rho},\delta_{2\rho})$ thus, for the angular region (\ref{eq:psi_rhoPi_cond}) to intersect the angular region (\ref{eq:argZ_cond}). Taking into consideration that $\psi_{\rho\lambda}$ can take arbitrary values satisfying the condition (\ref{eq:psi_rhoPi_cond}). We will use this arbitrariness and choose
\begin{equation}\label{eq:psiL=argZ}
  \psi_{\rho\lambda}=\arg z.
\end{equation}
Without any loss of generality we choose $|\lambda|=1$. As a result, the contour $\gamma_\rho(\varepsilon,\psi_{\rho\lambda},\delta_{1\rho},\delta_{2\rho})$ takes the form (see~Fig.~\ref{fig:loop_gammaRho_rhoPi_argZ})
\begin{equation}\label{eq:loop_gammaRho_argZ}
  \gamma_\rho(\varepsilon,\arg z,\delta_{1\rho},\delta_{2\rho})=
  \left\{\begin{array}{l}
    S_{1\rho}=\{u:\ \arg u=-\delta_{1\rho}+\arg z,\ |u|\geqslant\varepsilon^{1/\rho}\}, \\
    C_{\varepsilon\rho}=\{u:\ -\delta_{1\rho}+\arg z\leqslant\arg u\leqslant\delta_{2\rho}+\arg z,\  |u|=\varepsilon^{1/\rho}\},\\
    S_{2\rho}=\{u:\ \arg u= \delta_{2\rho}+\arg z,\ |u|\geqslant\varepsilon^{1/\rho}\},\\
    \end{array}\right.
\end{equation}
where $\tfrac{\pi}{2\rho}<\delta_{1\rho}\leqslant\min\left(\pi,\tfrac{\pi}{\rho}\right)$, $\tfrac{\pi}{2\rho}<\delta_{2\rho}\leqslant\min\left(\pi,\tfrac{\pi}{\rho}\right)$,
\begin{equation}\label{eq:z_cond}
    \tfrac{\pi}{2\rho}-\delta_{2\rho}+\pi< \arg z< -\tfrac{\pi}{2\rho}+\delta_{1\rho}+\pi.
\end{equation}
The formula (\ref{eq:MLF_tmp2}) will take the form
\begin{equation}\label{eq:MLF_tmp3}
  E_{\rho,\mu}(z)=\frac{\rho e^{-i\pi}}{2\pi i}\int_{\gamma_\rho(\varepsilon,\arg z,\delta_{1\rho},\delta_{2\rho})}\frac{\exp\left\{\left(u e^{-i\pi}\right)^{\rho}\right\} \left(ue^{-i\pi}\right)^{\rho(1-\mu)-1}}{1-\frac{z}{u}e^{i\pi}}du,\quad \lambda=e^{-i\rho\pi}.
\end{equation}
It should be noted that since we chose the value $\psi_{\rho\lambda}=\arg z$, then this automatically entailed the imposition on $\arg z$ the condition (\ref{eq:psi_rhoPi_cond}) which would lead to the condition (\ref{eq:z_cond}).

Now let us make a substitution in this expression of the variable of integration $u=\zeta ze^{i\pi}$. We have
\begin{equation*}
  \zeta=\frac{u}{z}e^{-i\pi}\quad\Rightarrow\quad |\zeta|=\frac{|u|}{|z|},\quad \arg\zeta=\arg u-\arg z-\pi.
\end{equation*}
 As a result of this substitution, the half-line $S_{1\rho}$ of the contour (\ref{eq:loop_gammaRho_argZ}) maps into the half-line $S_1=\{\zeta:\ \arg \zeta=-\delta_{1\rho}-\pi,\ |\zeta|\geqslant \varepsilon^{1/\rho}/|z|\}$ on the plane $\zeta$. It should be noted that with such a substitution of the integration variable, the condition (\ref{eq:uz_cond}) will take the form
\begin{equation*}
  \inf_{\zeta\in\gamma_\zeta(\epsilon,\arg z,\delta_{1\rho},\delta_{2\rho})}|\zeta|>1.
\end{equation*}
This means that the value $\varepsilon$ is chosen in such a way that $\varepsilon^{1/\rho}/|z|>1$. Introducing for convenience the notation \begin{equation}\label{eq:eps_|z|}
\varepsilon^{1/\rho}/|z|=1+\epsilon,\quad \epsilon>0
\end{equation}
we will obtain $S_1=\{\zeta:\ \arg \zeta=-\delta_{1\rho}-\pi,\ |\zeta|\geqslant 1+\epsilon\}$. Similarly, considering the mapping of the arc of the circle $C_{\varepsilon\rho}$ and half-line $S_{1\rho}$ from the complex plane $u$ into the complex plane $\zeta$ we get that the contour (\ref{eq:loop_gammaRho_argZ}) maps into the contour
\begin{equation*}
  \gamma_\zeta(\epsilon,\arg z,\delta_{1\rho},\delta_{2\rho})=
  \left\{\begin{array}{l}
    S_1=\{\zeta:\ \arg \zeta=-\delta_{1\rho}-\pi,\ |\zeta|\geqslant 1+\epsilon\}, \\
    C_\epsilon=\{\zeta:\ -\delta_{1\rho}-\pi \leqslant\arg \zeta \leqslant\delta_{2\rho}-\pi,\  |\zeta|=1+\epsilon\},\\
    S_2=\{\zeta:\ \arg \zeta= \delta_{2\rho}-\pi,\ |\zeta|\geqslant 1+\epsilon\},\\
    \end{array}\right.
\end{equation*}
where $\tfrac{\pi}{2\rho}<\delta_{1\rho}\leqslant\min\left(\pi,\tfrac{\pi}{\rho}\right)$, $\tfrac{\pi}{2\rho}<\delta_{2\rho}\leqslant\min\left(\pi,\tfrac{\pi}{\rho}\right)$, $\tfrac{\pi}{2\rho}-\delta_{2\rho}+\pi< \arg z< -\tfrac{\pi}{2\rho}+\delta_{1\rho}+\pi$.
And finally, making a substitution of the variable of integration in  (\ref{eq:MLF_tmp3}) we get
\begin{equation*}
  E_{\rho,\mu}(z)=\frac{\rho}{2\pi i}\int_{\gamma_\zeta(\epsilon,\arg z,\delta_{1\rho},\delta_{2\rho})}\frac{\exp\left\{\left(z\zeta\right)^{\rho}\right\} \left(z\zeta\right)^{\rho(1-\mu)}}{\zeta-1}d\zeta,
\end{equation*}
\begin{flushright}
  $\Box$
\end{flushright}
\end{proof}

\section{Conclusion}

As we noted in the Introduction, the main difference between the existing integral representations of the Mittag-Leffler function consists in the integration contour. Each subsequent representation took account of additional properties of the Hankel contour in the integral representation of the gamma function. In this paper, it has been shown that the change in the directions of the half-lines of the Hankel contour can be represented as the rotation of this contour in the complex plane (see lemma~\ref{lemma_GF}). This presentation is much more convenient to use, as it is more illustrative. However, the condition of lemma~\ref{lemma_GF} constrains the rotation angle by the condition (\ref{eq:psiCond}). The use of corollary~\ref{coroll:GF_lambda} allows one to expand this range to the entire complex plane. This corollary gives an opportunity us to interpret the rotation of the Hankel contour as the sum of two independent rotations: the rotation of the complex plane as a whole by an angle $\arg\lambda$ and rotation of the Hankel contour on this complex plane by an angle  $\psi_\lambda$.

This property of the Hankel contour turned out to be useful in deriving the integral representation of the function $E_{\rho,\mu}(z)$. It gave an opportunity to associate the position of the Hankel contour with the value $\arg z$. As a result, a modified form of the integral representation of the Mittag-Leffler function is obtained. The advantage of this integral representation is that its singular points have a fixed location on the complex plane.  As we can see from (\ref{eq:MLF_int}), this representation has two singular points $\zeta=1$ and $\zeta=0$. The point $\zeta=1$  is a pole of the first order and the point $\zeta=0$ depending on the parameter value $\rho$ can be either a pole, a branch point, or a regular point. The fact that the singular points have a fixed location somewhat simplifies further study of this integral representation. For example, in the work \cite{Saenko2020d}  when passing in the representation (\ref{eq:MLF_int}) from the contour integral to integrals over real variables, it was possible to find that the resulting representations of the Mittag-Leffler function can be written in two forms. The disadvantage of the representation (\ref{eq:MLF_int}) is that this representation is valid only for the range of values  $\arg z$, satisfying the condition (\ref{eq:z_cond_lemm}). This constraint appeared as a result of linking the integration contour with  $\arg z$ (see (\ref{eq:psiL=argZ})) and it somewhat restricts the applicability of the obtained representation. However, if we use  corollary~\ref{coroll:GF_lambda}, then it is possible to generalize the representation (\ref{eq:MLF_int}) so that it will be valid on the entire complex plane.

\section*{Acknowledgments}
This work was supported by the Russian Foundation for Basic Research. (grant \No~19-44-730005 and 20-07-00655).

The author thanks to M. Yu. Dudikov for translation the article into English.

\bibliographystyle{elsarticle-num}
\bibliography{d:/bibliography/library}

\end{document}